\RequirePackage{fix-cm}
\documentclass[smallextended]{svjour3}  
\smartqed 
\usepackage{filecontents}
\usepackage{graphicx}
\usepackage{enumitem}
\usepackage{amsmath}
\usepackage{amssymb}
\usepackage{pifont}
\usepackage{theorem} 
\usepackage{euscript}
\usepackage{exscale,relsize}
\usepackage{epic,eepic}
\usepackage{multicol}
\usepackage{makeidx}
\usepackage{srcltx}   
\usepackage{xcolor}
\usepackage{url}
\usepackage{charter}
\usepackage{mathrsfs} 
\usepackage{graphicx}
\usepackage{authblk}
\usepackage{tikz}
\usepackage[normalem]{ulem}  
\usepackage[usenames,dvipsnames]{pstricks}
\usepackage{pstricks-add}
\usepackage{pst-grad}
\usepackage{pst-plot} 
\usepackage{pst-eucl} 

\definecolor{labelkey}{rgb}{0,0.08,0.45}
\definecolor{refkey}{rgb}{0,0.6,0.0}
\definecolor{Brown}{rgb}{0.45,0.0,0.05}
\definecolor{dbrown}{rgb}{0.50,0.25,0.00}
\definecolor{dgreen}{rgb}{0.00,0.49,0.00}
\definecolor{dblue}{rgb}{0,0.08,0.75}
\RequirePackage[colorlinks,hyperindex]{hyperref}
\hypersetup{linktocpage=true,citecolor=dblue,linkcolor=dgreen}
\renewcommand{\leq}{\ensuremath{\leqslant}}
\renewcommand{\geq}{\ensuremath{\geqslant}}

\newcommand{\minimize}[2]{\ensuremath{\underset{\substack{{#1}}}%
{\text{minimize}}\;\;#2 }}

\newcommand{\menge}[2]{\big\{{#1}~\big |~{#2}\big\}} 
 
\newcommand{\HHH}{\ensuremath{\boldsymbol{\mathcal H}}}

\newcommand{\Sum}{\ensuremath{\displaystyle\sum}}
\newcommand{\emp}{\ensuremath{{\varnothing}}}
\newcommand{\Id}{\ensuremath{\operatorname{Id}}\,}

\newcommand{\RR}{\ensuremath{\mathbb{R}}}
\newcommand{\RP}{\ensuremath{\left[0,+\infty\right[}}
\newcommand{\RO}{\ensuremath{\left]1,+\infty\right[}}

\newcommand{\RPP}{\ensuremath{\left]0,+\infty\right[}}

\newcommand{\RX}{\ensuremath{\left]-\infty,+\infty\right]}}

\newcommand{\NN}{\ensuremath{\mathbb N}}

\newcommand{\pinf}{\ensuremath{{+\infty}}}

\newcommand{\dom}{\ensuremath{\text{\rm dom}\,}}

\newcommand{\prox}{\ensuremath{\text{\rm prox}}}
\newcommand{\proj}{\ensuremath{\text{\rm proj}}}

\newcommand{\sign}{\ensuremath{\text{\rm sign}}}

\newcommand{\reli}{\ensuremath{\text{\rm ri}\,}}

\newcommand{\zeroun}{\ensuremath{\left]0,1\right[}}

\theoremstyle{plain}{\theorembodyfont{\rmfamily}%
\newtheorem{algorithm}[theorem]{Algorithm}}
\theoremstyle{plain}{\theorembodyfont{\rmfamily}%
\newtheorem{model}[theorem]{Model}}
\renewcommand\theenumi{(\roman{enumi})}
\renewcommand\theenumii{(\alph{enumii})}

\begin{document}

\title{Regression models for compositional data: 
General log-contrast formulations, proximal optimization, 
and microbiome data applications 
\thanks{The work of P. L. Combettes was supported by the 
National Science Foundation under grant DMS-1818946.}
}

\titlerunning{Log-contrast regression for compositional 
data analysis}  

\author{Patrick L. Combettes  \and Christian L. M\"uller 
}


\institute{P. L. Combettes \at
Department of Mathematics, North Carolina State University,
Raleigh, NC 27695-8205, USA\\
Tel.: +1 919 515 2671\\
\email{plc@math.ncsu.edu}   
\and
C. L. M\"uller \at
Center for Computational Mathematics, Flatiron Institute, New York,
NY 10010, USA \\
Tel.: +1 646 603 2716\\
\email{cmueller@flatironinstitute.org}
}

\date{Received: date / Accepted: date}

\maketitle

\begin{abstract}
Compositional data sets are ubiquitous in science, including 
geology, ecology, and microbiology. 
In microbiome research, compositional data primarily 
arise from high-throughput sequence-based profiling 
experiments. These data comprise microbial compositions in 
their natural habitat and are often paired with covariate 
measurements 
that characterize physicochemical habitat properties or
the physiology of the host. Inferring parsimonious statistical
associations between microbial compositions and habitat- or
host-specific covariate data is an important step in exploratory
data analysis. A standard statistical model linking compositional
covariates to continuous outcomes is the linear log-contrast model.
This model describes the response as a linear combination of
log-ratios of the original compositions and has been extended 
to the high-dimensional setting via regularization. In
this contribution, we propose a general convex optimization model 
for linear log-contrast regression which includes many previous
proposals as special cases. We introduce a proximal algorithm that
solves the resulting constrained optimization problem exactly
with rigorous convergence guarantees. 
We illustrate the versatility of our approach by
investigating the performance of several model instances on
soil and gut microbiome data analysis tasks. 

\end{abstract}

\section{Introduction}
\label{intro}
Compositional data sets are ubiquitous in many areas of science,
spanning such disparate fields as geology and ecology. In
microbiology, compositional data arise from high-throughput
sequence-based microbiome profiling techniques, such as 
targeted amplicon
sequencing (TAS) and metagenomic profiling. These methods generate 
large-scale genomic survey data of microbial community compositions 
in their natural habitat, ranging from marine ecosystems to 
host-associated environments. Elaborate bioinformatics processing 
tools \cite{Callahan2015,Caporaso2010,Edgar2013,%
Lagkouvardos2017,Schloss2009} 
typically summarize TAS-based sequencing reads into sparse
compositional counts of operational taxonomic units (OTUs). The
quantification of the relative abundances of OTUs in the
environment is often accompanied by measurements of other
covariates, including physicochemical properties of the underlying
habitats, variables related to the health status of the host, or
those coming from other high-throughput protocols, such as
metabolomics or flow cytometry.

An important step in initial exploratory analysis of such data sets
is to infer parsimonious and robust statistical relationships
between the microbial compositions and habitat- or host-specific
measurements. Standard linear regression modeling cannot be
applied in this context because the microbial count data carry
only relative or compositional information. One of the most popular
approaches to regression modeling with compositional covariates is
the log-contrast regression model, originally proposed in
\cite{Aitchison1984} in the context of experiments with mixtures.
The linear log-contrast model expresses the continuous outcome of
interest as a linear combination of the log-transformed
compositions subject to a zero-sum constraint on the regression
vector. This leads to the intuitive interpretation of the response as
a linear combination of log-ratios of the original compositions. In
a series of papers, the linear log-contrast model 
has been generalized to the 
high-dimensional setting via regularization. 
The sparse linear log-contrast model, introduced in \cite{Lin2014}, 
considers variable selection via $\ell^1$ regularization and has been
extended (i) to multiple linear constraints for sub-compositional
coherence across predefined groups of predictors \cite{Shi2016};
(ii) to sub-composition selection via tree-structured
sparsity-inducing penalties \cite{Wang2017}; (iii) to longitudinal
data modeling via a constraint group lasso penalty \cite{Sun2018};
and (iv) to outlier detection via a mean shift modeling approach
\cite{Mishra2019}. A common theme of these statistical approaches to 
log-contrast modeling is the formulation of the estimators as the 
solution of a convex optimization problem, and the 
theoretical analysis of the statistical properties of these
estimators under suitable assumptions on the data. 

In the present paper, we take a complementary approach and focus on
the structure of the optimization problems underlying log-contrast
modeling. We propose an general optimization model for linear
log-contrast regression which includes previous proposals as
special cases and allows for a number of novel formulations that
have not yet been explored. A particular feature of our model is
the joint estimation of regression vectors and associated scales
for log-contrast models, similar to the scaled Lasso approach in
high-dimensional linear regression \cite{Sun2012}. This is achieved
by leveraging recent results on the connection between
perspective functions and statistical models
\cite{Svva18,Combettes2018a,Combettes2018b}. We introduce a
Douglas-Rachford splitting algorithm that produces an exact solution 
to the resulting constrained optimization problems with 
rigorous guarantees on the convergence of the iterates. 
By contrast, most existing approaches to solve such problems 
proceed by first approximating it and then employing 
coordinate descent methods with less demanding convergence
guarantees. We illustrate the versatility
of our modeling approach by applying novel log-contrast model
instances to environmental and gut microbiome data analysis tasks.

\section{Linear log-contrast models}
\label{sec:llm}
We first introduce the statistical log-contrast data formation
model under consideration. We then review several prominent
estimators for regularized log-contrast regression models. 

\subsection{Statistical log-contrast data formation model}
\label{subsec:fwrdm}

Let $Z$ be a known $(n\times p)$-dimensional compositional design 
matrix with rows $(z_i)_{1\leq i\leq n}$ in the simplex 
$\menge{(\zeta_1,\ldots,\zeta_p)\in\RPP^p}{\sum_{k=1}^p\zeta_k=1}$.
In the microbiome context, each
row represents a composition of $p$ OTUs or components at a higher
taxonomic rank. We apply a log transform $(x_i)_{1\leq i\leq n}=
(\log{z_i})_{1\leq i\leq n}$ resulting in the design matrix 
$X\in\RR^{n\times p}$. In this context, we introduce the following
log-contrast data formation model. 

\begin{model}
\label{m:1}
The vector $y=(\eta_i)_{1\leq i\leq n}\in\RR^n$ of observations is 
\begin{equation}
\label{llcm:1}
y=X\overline{b}+\overline{o}+Se,\quad\text{with}\quad 
C^{\top}\overline{b}=0, 
\end{equation}
where $X\in\RR^{n\times p}$ is the aforementioned design matrix
with rows
$(x_i)_{1\leq i\leq n}$, $\overline{b}\in\RR^p$ is the unknown
regression vector (location), $\overline{o}\in\RR^n$ is the unknown
mean shift vector containing outliers, $e\in\RR^n$ is a vector of
realizations
of i.i.d. zero mean random variables, $S\in\RP^{n\times n}$ is
a diagonal matrix the diagonal of which are the (unknown)
standard deviations, and $C\in\RR^{p\times K}$ is
a matrix expressing $K$ linear constraints on the regression vector. 
\end{model}

The linear log-contrast data formation model is similar to the
standard (heteroscedastic) linear model with the important
difference that there are linear equality constraints on the
regression vector. This stems from the fact that the entries in
$X\in\RR^{n\times p}$ are not independent due to the compositional
nature. In the original model \cite{Aitchison1984}, the constraint
matrix $C\in\RR^{p\times K}$ is the $p$-dimensional all-ones
vector $\mathbf{1}_p$, resulting in a zero-sum constraint on the
regression vector. To gain some intuition about the implications
of this constraint, consider a two-dimensional example with given
estimates $b=(\beta_1,\beta_2)$, and denote by $\xi_{i,1}$ and
$\xi_{i,2}$ the first and second column entries of $X$. The linear
equality constraint enforces $\beta_2=-\beta_1$, and thus each
observation can be expressed as
\begin{equation}
\label{llcm:2a}
\eta_i= \beta_1 \xi_{i,1}-\beta_1 \xi_{i,2} \,.
\end{equation}
Due to the construction of the design matrix as the log
transformation of the compositions, this model is equivalent to
\begin{equation}
\label{llcm:2b}
\eta_i= \beta_1 \log{\zeta_{i,1}}-\beta_1 \log{\zeta_{i,2}}=
\beta_1 \log{\frac{\zeta_{i,1}}{\zeta_{i,2}}} \,,
\end{equation}
which expresses the response as a linear function of the log-ratios
of the original compositional components. This example also shows
that the regression coefficients in the log-contrast model bear a
different interpretation than in the standard linear model.
Combined log-ratio coefficients relate the response to log-fold
changes of the corresponding component ratios. 

\subsection{Statistical estimators for log-contrast models}

\subsubsection{Sparse log-contrast regression}
In the low-dimensional setting, the standard log-contrast model
with zero-sum constraints can be estimated by solving a
least-squares problem subject to a linear constraint, or
alternatively, via standard linear regression applied to
isometrically log-ratio transformed compositions \cite{Hron2012}.
In the high-dimensional setting, we need structural
assumptions on the regression vector for consistent estimation. To
this end, the sparse log-contrast model was introduced in
\cite{Lin2014}. It is based on the optimization problem
\begin{equation}
\label{e:lin14}
\minimize{\substack{b\in\RR^p\\ {\sum_{k=1}^p}
\beta_k=0}}{\frac{1}{2n}\|X b-y\|_2^2+\lambda \|b\|_1},
\end{equation}
where $\|\cdot\|_1$ is the $\ell^1$ norm and $\lambda\in\RP$ is a
tuning parameter that balances model fit and sparsity of the
solution. The estimator enjoys several desirable properties,
including scale invariance, permutation invariance, and selection
invariance. The latter property is intimately related to the
principle of sub-compositional coherence \cite{Aitchison1986} and
means that the estimator is unchanged if one knew in advance the
sparsity pattern of the solution and applied the procedure to the
sub-compositions formed by the nonzero components. In
\cite{Lin2014}, model consistency guarantees are derived for
the estimator and the underlying optimization
problem is approached via penalization. The proposed iterative
algorithm alternates between estimating the Lagrange multipliers
and solving a convex subproblem with a coordinate descent strategy.
Model selection for the regularization parameter $\lambda$ is
performed with a generalized information criterion. 

\subsubsection{Sparse log-contrast regression with side information}
In many situations, it is desirable to incorporate side information
about the covariates into log-contrast modeling. For instance, for
microbial compositions, each component can be associated with
taxonomic or phylogenetic information, thus relating the $p$
components through a rooted taxonomic or phylogenetic tree
$\mathcal{T}_{p}$. One way to use this hierarchical tree
information is to perform log-contrast regression at a higher
taxonomic level, effectively reducing the dimensionality of the
regression problem. Let $\mathcal{T}_{p}$ be a tree with $1\leq
i_h \leq h$ levels and $p$ leaves and assume that, at given level
$i_h$, the $p$ compositions split into $K$ groups with sizes
$(p_k)_{1\leq k \leq K}$. Sub-compositional coherence across the
groups can be expressed by the linear constraints $C^{\top}b=0$,
where $C$ is an orthogonal $(p\times K)$-dimensional binary matrix.
The $k^{th}$ column comprises $p_k$ ones at the coordinates of the
components that belong to the $k^{th}$ group. Sparse log-contrast
regression with group-level compositional coherence can thus be
achieved by solving the optimization problem
\begin{equation}
\label{e:shi16}
\minimize{\substack{b\in\RR^p\\ C^{\top}b=0}}
{\frac{1}{2n}\|X b-y\|_2^2+\lambda\|b\|_1},
\end{equation}
where $\lambda\in\RP$ is a tuning parameter. In \cite{Shi2016},
model consistency guarantees are derived for this estimator
as well as a debiasing procedure for the final estimates. This is
done by extending results from \cite{Javanmard2014} to the
log-contrast 
setting. In \cite{Lin2014}, the underlying
optimization problem is approached via an augmented 
Lagrangian approach,
while model selection is achieved by scaling a theoretically
derived $\lambda_0$ with a data-driven heuristic estimate of the
standard deviation $\sigma$ \cite{Sun2012}, resulting in $\lambda
=\lambda_0\sigma$. 

An alternative way of incorporating tree
information has been proposed in \cite{Wang2017}. There, the tree
structure is encoded in a parameterized matrix $J_{\alpha} \in
\RR^{{m-1}\times p}$, where $m$ is the number of vertices in the
tree. An estimator based on the minimization problem
\begin{equation}
\label{e:wang17}
\minimize{\substack{b\in\RR^p\\ \sum_{k=1}^p
\beta_k=0}}{\frac{1}{2n}\|X b-y\|_2^2+\lambda
\|J_{\alpha} b\|_1}
\end{equation}
is proposed, where $\lambda\in\RP$ is a tuning parameter. 
The structure of
$J_{\alpha}$ promotes tree-guided sparse sub-composition selection
and comprises a weighting parameter $\alpha \in [0,1]$. The authors
of \cite{Wang2017}
are unable to solve the optimization in \eqref{e:wang17}
exactly and resort to a heuristic that abandons the linear
constraints and solves a generalized Lasso problem instead. The two
tuning parameters $\lambda$ and $\alpha$ are selected via an
information criterion. 

\subsubsection{Robust log-contrast regression}
The previous estimators assume the response to be outlier-free with
respect to the statistical model under consideration. One way to
relax this assumption and to guard against outliers in the response
is to use a robust data fitting term. In \cite{Mishra2019}, the
robust log-contrast regression is introduced via mean shift
modeling; see, e.g., \cite{Antoniadis2007,She2011}. One specific
instance of this framework considers the estimation
problem
\begin{equation}
\label{e:mishra19}
\minimize{b\in\RR^p,\, o\in\RR^n} {\frac{1}{2n}\|X b-y-o\|_2^2 +
\lambda_1 \|b\|_1+\lambda_2 \|o\|_1}, \quad\text{where}\quad
C^{\top}b=0,
\end{equation}
and where nonzero elements in the mean shift vector $o\in\RR^n$
capture outlier data, and $\lambda_1$ and $\lambda_2$ are tuning
parameters. In \cite{Nguyen2013a}, the objective function in
\eqref{e:mishra19} is approximated in the form of
\eqref{e:shi16} with a single tuning parameter. 
As shown in \cite{Antoniadis2007} for partial
linear models and \cite{She2011} for outlier detection, an
equivalent form of \eqref{e:mishra19} is to use the Huber
function \cite{Huber1964} as robust data fitting function and the
$\ell^1$ norm as regularizer. The Huber function is defined as
\begin{equation}
\label{e:huber}
h_{\rho}\colon\RR\to\RR\colon u\mapsto
\begin{cases}
\rho|u|-\dfrac{\rho^2}{2},&\text{if}\;\;|u|>\rho;\; \\
\dfrac{|u|^2}{2}, &\text{if}\;\; |u|\leq\rho\,,\\
\end{cases}
\end{equation}
where $\rho \in \RO$ is a fixed parameter with default value
$\rho=1.345$ that determines the transition from the quadratic to
the linear part. The model in \eqref{e:mishra19} can be
written as 
\begin{equation}
\label{e:pcm19}
\minimize{\substack{b\in\RR^p\\ C^{\top}b=0}}
{\frac{1}{2n}\sum_{i=1}^{n} h_{\rho}(x_i b-\eta_i)
+\lambda_1 \|b\|_1}.
\end{equation}
After model estimation, each data point in the linear region of the
Huber function is considered an outlier. The latter two models
thus allow for joint sparse selection of predictors and outliers in a
convex framework.

\section{Optimization of general log-contrast models}
\label{sec:3}
We introduce an optimization model for general log-contrast
regression that includes all previous examples as special cases. We
assume that the data follow the data formation model outlined in
Model~\ref{m:1}. Our model belongs to the class of perspective
M-estimation models \cite{Combettes2018b} and allows for joint
estimation of regression parameters and corresponding scales while
preserving the overall convexity of the model. We then present a
proximal algorithm that can solve instances of the
optimization model with theoretical guarantees on the convergence
of the iterates. Finally, we propose two model selection schemes for
practical regularization parameter selection that leverage the joint
scale estimation capability of our optimization model.

\subsection{Convex optimization model}
\label{subsec:optmodel}
Let us first introduce some notation (see \cite{Livre1,Rock70} for
details). We denote by $\Gamma_0(\RR^n)$ the class of lower 
semicontinuous convex functions $\varphi\colon\RR^n\to\RX$ such 
that $\dom\varphi=\menge{x\in\RR^n}{\varphi(x)<\pinf}\neq\emp$.
Given $\varphi\in\Gamma_0(\RR^n)$ and $x\in\RR^n$, the unique 
minimizer of $\varphi+\|x-\cdot\|_2^2/2$ is
denoted by $\prox_\varphi x$.
Now let $D$ be a convex subset of $\RR^n$. Then
$\iota_D$ is the indicator function of $D$ (it takes values 
$0$ on $D$ and $\pinf$ on its complement), $\reli D$ is the relative
interior of $D$ (its interior relative to its affine hull), and, if
$D$ is nonempty and closed, $\proj_D=\prox_{\iota_D}$ is the
projection operator onto $D$. 

The following general log-contrast optimization model enables the
joint estimation of the regression vector
$\overline{b}=(\overline{\beta}_k)_{1\leq k\leq p}\in\RR^p$ and of
the scale vector 
$\overline{s}=(\overline{\sigma}_i)_{1\leq i\leq N}\in\RR^N$
in Model~\ref{m:1} within a convex optimization setting.

\begin{problem}
\label{prob:1}
Consider the setting of Model~\ref{m:1}.
Let $N$ and $M$ be strictly positive integers,
let $D$ be a vector subspace of $\RR^N$,
let $(n_i)_{1\leq i\leq N}$
be strictly positive integers such that $\sum_{i=1}^Nn_i=n$, 
let $(m_i)_{1\leq i\leq M}$ be strictly positive integers, and set
$m=\sum_{i=1}^M{m_i}$.
For every $i\in\{1,\ldots,N\}$, let 
$\varphi_i\in\Gamma_0(\RR^{n_i})$, let 
\begin{equation}
\label{e:perspective}
\begin{array}{rl}
\widetilde{\varphi}_i\colon
\RR\times\RR^{n_i}&\to\RX\\[2mm]
(\sigma_i,u_i)~~&\mapsto
\begin{cases}
\sigma_i\varphi_i(u_i/\sigma_i),&\text{if}\;\:\sigma_i>0;\\
\displaystyle{\sup_{u\in\dom\varphi_i}}
\big(\varphi_i(u+u_i)-\varphi_i(u)\big),
&\text{if}\;\:\sigma_i=0;\\
\pinf,&\text{if}\;\;\sigma_i<0
\end{cases}
\end{array}
\end{equation}
be the perspective of $\varphi_i$,
let $X_i\in\RR^{n_i\times p}$,
and let $y_i\in\RR^{n_i}$ be such that 
\begin{equation}
\label{e:blocks}
X=
\begin{bmatrix}
X_1\\
\vdots\\
X_N\\
\end{bmatrix}
\quad\text{and}\quad
y=
\begin{bmatrix}
y_1\\
\vdots\\
y_N\\
\end{bmatrix}.
\end{equation}
Finally, set
\begin{equation}
\label{e:E}
E=\menge{b\in\RR^p}{C^\top b=0}
\end{equation}
and, for every $i\in\{1,\ldots,M\}$, let 
$\psi_i\in\Gamma_0(\RR^{m_i})$ and 
$L_i\in\RR^{m_i\times p}$.
The objective is to
\begin{equation}
\label{e:prob0}
\minimize{s\in D,\:b\in E}
{\Sum_{i=1}^{N}\widetilde{\varphi}_i
\big(\sigma_i,X_ib-y_i\big)+\Sum_{i=1}^{M}
\psi_i\big(L_i b\big)}.
\end{equation}
\end{problem}

\begin{remark}
\label{r:prob1}
Problem~\ref{prob:1} comprises four main components which
are associated with different aspects of the general log-contrast 
regression model. 
\begin{itemize}
\item
The perspective functions $(\widetilde{\varphi}_i)_{1\leq i\leq N}$
play the role of the loss function in statistical estimation and
couple the estimation of the scale vector $s$ and the regression
vector $b$. Because the functions $({\varphi}_i)_{1\leq i\leq N}$
are convex, the overall minimization problem \eqref{e:prob0}
remains a convex one in $(s,b)$.
\item
Problem~\ref{prob:1} allows for the partitioning of the design matrix
$X$ and response $y$ into $N$ blocks with individual scale
parameters $(\sigma_i)_{1\leq i\leq N}$. This is beneficial when
data from multiple measurement sources are available for the
prediction of the response or when heteroscedasticity in the design
matrix is expected for different groups of measurements.  
Introducing multiple scales has also numerical
advantages. Indeed, as discussed in \cite{Combettes2018b}, 
certain proximity operators of perspective functions are easier 
to compute in separable form.
\item
The vector subspaces $D$ and $E$ (see \eqref{e:E}) enforce linear 
constraints on the scale vector $s=(\sigma_i)_{1\leq i\leq N}$ 
and the regression vector $b$, respectively.
\item
Additional properties of the regression vector, such as (structured) 
sparsity, are promoted through the use of the penalization functions 
$(\psi_i)_{1\leq i\leq M}$ and the matrices $(L_i)_{1\leq i\leq M}$.
The penalization functions typically contain a free parameter
$\lambda$ the setting of which requires a model selection
strategy. 
\end{itemize}
\end{remark}

Perspective functions are discussed in
\cite{Livre1,Svva18,Combettes2018a,Combettes2018b,Rock70}. The 
construction \eqref{e:perspective} guarantees that 
$(\forall i\in\{1,\ldots,N\})$
$\widetilde{\varphi}_i\in\Gamma_0(\RR^{n_i})$.
We provide below two examples of perspective functions that will be
used in the numerical investigations of Section~\ref{sec:4}.

\begin{figure}[h]
\centering
\includegraphics[width=12cm]{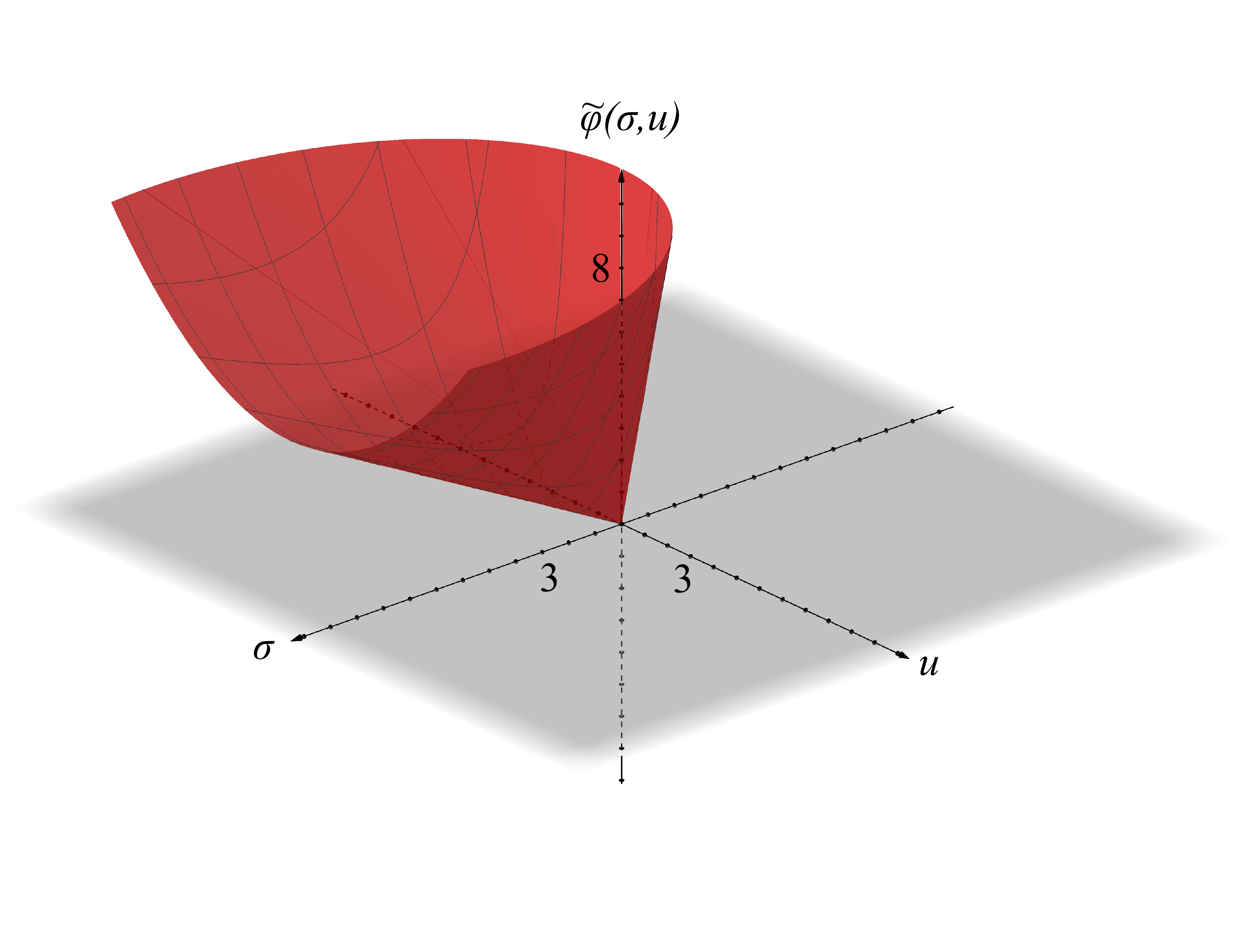}
\caption{Perspective of $\varphi=|\cdot|^2+1/2$.}
\label{fig:1}
\end{figure}

\begin{example}
\label{ex:1}
Consider the function $\varphi=\|\cdot\|_2^2+1/2$ defined on
the standard Euclidean space $\RR^P$. Then 
\eqref{e:perspective} yields (see Fig.~\ref{fig:1})
\begin{equation}
\label{e:perspective1}
\begin{array}{rl}
\widetilde{\varphi}\colon
\RR\times\RR^P&\to\RX\\[2mm]
(\sigma,u)~~&\mapsto
\begin{cases}
\dfrac{\sigma}{2}+\dfrac{\|u\|_2^2}{\sigma},
&\text{if}\;\:\sigma>0;\\
0,&\text{if}\;\:\sigma=0\;\:\text{and}\;\:u=0;\\
\pinf,&\text{otherwise.}
\end{cases}
\end{array}
\end{equation}
Now fix $(\sigma,u)\in\RR\times\RR^P$ and $\gamma\in\RPP$.
If $4\gamma\sigma+\|u\|_2^2>2\gamma^2$, 
let $t$ be the unique solution in $\RPP$ to the equation 
\begin{equation}
\label{e:10f}
\gamma t^3+2(2\sigma+3\gamma)t-8\|u\|_2=0,
\end{equation}
and set $p=tu/\|u\|_2$ if $u\neq 0$, and $p=0$ if $u=0$.
Then \cite[Example~2.4]{Combettes2018b} yields
\begin{equation}
\label{e:10b}
\prox_{\gamma\widetilde{\varphi}}(\sigma,u)=
\begin{cases}
\bigg(\sigma+\dfrac{\gamma}{2}\bigg(\dfrac{t^2}{2}-1\bigg),
u-\gamma p\bigg),
&\text{if}\;\;4\gamma\sigma+\|u\|_2^2>2\gamma^2;\\
(0,0),&\text{if}\;\;4\gamma\sigma+\|u\|_2^2\leq2\gamma^2.
\end{cases}
\end{equation}
\end{example}

\begin{figure}[h]
\centering
\includegraphics[width=12cm]{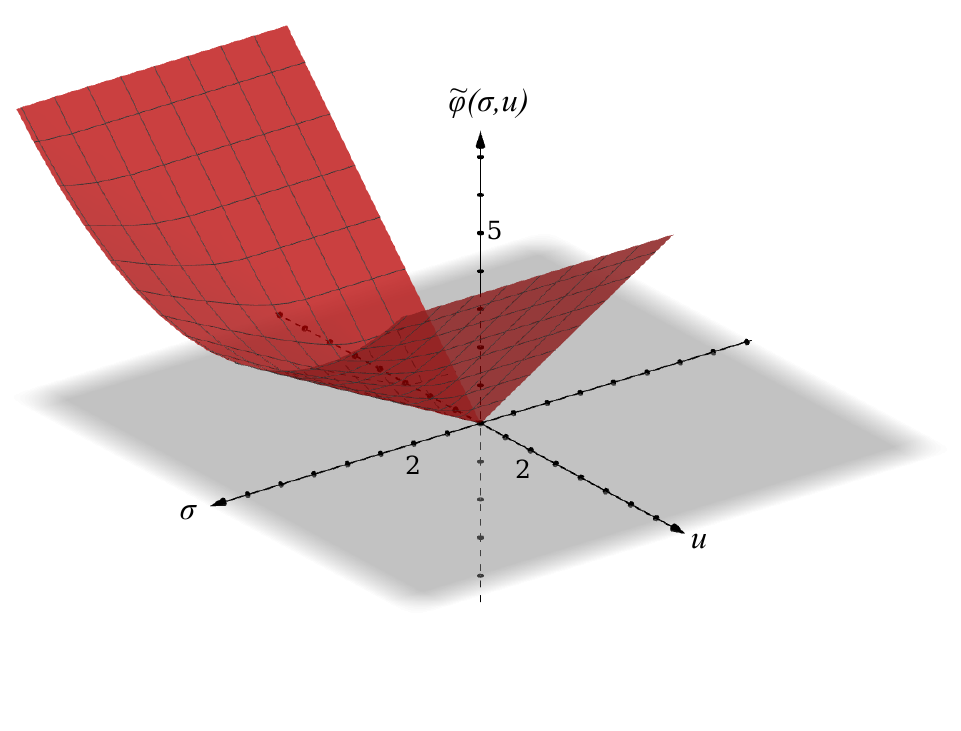}
\caption{Perspective of 
$\varphi=h_1+1/2$, where $h_1$ is the Huber function.}
\label{fig:2}
\end{figure}

A prominent estimator where the perspective function
\eqref{e:perspective1} is used as a loss function in conjunction
with the $\ell^1$ norm as penalization function is the scaled Lasso
estimator for high-dimensional sparse linear regression
\cite{Sun2012}. 

\begin{example}
\label{ex:2}
Set $\varphi=h_1+1/2$, where $h_1$ is the Huber function of 
\eqref{e:huber}. Then \eqref{e:perspective} yields 
(see Fig.~\ref{fig:2})
\begin{equation}
\label{e:perspective2}
\begin{array}{rl}
\widetilde{\varphi}\colon
\RR\times\RR&\to\RX\\[2mm]
(\sigma,u)&\mapsto
\begin{cases}
\dfrac{(1-\rho^2)\sigma}{2}+\rho|u|,
&\text{if}\;\:|u|>\sigma\rho\;\text{and}\;\sigma>0;\\[2mm]
\dfrac{\sigma}{2}+\dfrac{|u|^2}{2\sigma},
&\text{if}\;\:|u|\leq\sigma\rho\;\text{and}\;\sigma>0;\\
\rho|u|,&\text{if}\;\:\sigma=0;\\
\pinf,&\text{if}\;\;\sigma<0.
\end{cases}
\end{array}
\end{equation}
Now fix $(\sigma,u)\in\RR\times\RR$ and $\gamma\in\RPP$. 
Then \cite[Example~2.5]{Combettes2018b} asserts that
$\prox_{\gamma\widetilde{\varphi}}(\sigma,u)$ is
computed as follows.
\begin{enumerate}
\item
\label{ex:ghi}
Suppose that $|u|\leq\gamma\rho$ and 
$|u|^2\leq\gamma(\gamma-2\sigma)$. Then 
$\prox_{\gamma\widetilde{\varphi}}(\sigma,u)=(0,0)$.
\item
\label{ex:ghii}
Suppose that $\sigma\leq\gamma(1-\rho^2)/2$ and 
$|u|>\gamma\rho$. Then 
\begin{equation}
\label{e:gh3}
\prox_{\gamma\widetilde{\varphi}}(\sigma,u)=
\Bigg(0,\bigg(1-\dfrac{\gamma\rho}{|u|}\bigg)u\Bigg).
\end{equation}
\item
\label{ex:ghiii}
Suppose that $\sigma>\gamma(1-\rho^2)/2$ and 
$|u|\geq\rho\sigma+\gamma\rho(1+\rho^2)/2$. Then 
\begin{equation}
\label{e:gh4}
\prox_{\gamma\widetilde{\varphi}}(\sigma,u)=
\Bigg(\sigma+\dfrac{\gamma}{2}\big(\rho^2-1\big),
\bigg(1-\dfrac{\gamma\rho}{|u|}\bigg)u\Bigg).
\end{equation}
\item
\label{ex:ghiv}
Suppose that
$|u|^{2}>\gamma(\gamma-2\sigma)$ and 
$|u|<\rho\sigma+\gamma\rho(1+\rho^2)/2$.
If $u\neq 0$, let $t$ be the unique
solution in $\RPP$ to the equation
\begin{equation}
\label{e:hb11}
\gamma t^3+(2\sigma+\gamma)t-2|u|=0.
\end{equation}
Then 
\begin{equation}
\label{e:gh12}
\prox_{\gamma\widetilde{\varphi}}(\sigma,u)=
\begin{cases}
\big(\sigma+\gamma(t^2-1)/2,u-\gamma t\sign(u)\big),
&\text{if}\;\;2\gamma\sigma+|u|^2
>\gamma^2;\\
\big(0,0\big),&\text{if}\;\;2\gamma\sigma+
|u|^2\leq\gamma^2.
\end{cases}
\end{equation}
\end{enumerate}
\end{example}

Using the perspective function \eqref{e:perspective2} as a loss
function and the $\ell^1$ norm as a penalization function recovers
a robust version of the scaled Lasso approach 
\cite{Combettes2018b,Owen2007}.

\subsection{Algorithm}
Our algorithmic solution method to solve Problem~\ref{prob:1}
is patterned after that of \cite{Combettes2018b}, which relies on
an application of the Douglas-Rachford splitting algorithm in a
higher-dimensional space. To present the algorithm and its
convergence properties, it is convenient to introduce the matrices
\begin{equation}
\label{e:nancago12}
A=
\begin{bmatrix}
X_1\\
\vdots\\
X_N\\
L_1\\
\vdots\\
L_M
\end{bmatrix},
\quad
Q=A^\top(\Id+{AA}^\top)^{-1},
\quad\text{and}\quad
W=\Id-C(C^\top C)^{-1}C^\top,
\end{equation}
together with the function
\begin{equation}
\label{e:g}
\begin{array}{rl}
\mathsf{g}\colon
&\RR^N\times\RR^{n_1}\times\cdots\times\RR^{n_N}
\times\RR^{m_1}\times\cdots\times\RR^{m_M}\to\RX\\
&\big(s,u_1,\ldots,u_N,v_1,\ldots,v_M\big)
\mapsto{\Sum_{i=1}^{N}\widetilde{\varphi}_i
(\sigma_i,u_i-y_i)+\Sum_{i=1}^{M}\psi_i(v_i)},
\end{array}
\end{equation}
and to define, for every iteration index $k\in\NN$, the vectors
\begin{equation}
\label{e:trenet}
\begin{cases}
s_{k}=(\sigma_{1,k},\ldots,\sigma_{N,k})\in\RR^N\\
h_{s,k}=(\eta_{1,k},\ldots,\eta_{N,k})\in\RR^N\\
h_{b,k}=(h_{1,k},\ldots,h_{N,k},h_{N+1,k},\ldots,h_{N+M,k})\\
\hfill\in
\RR^{n_1}\times\cdots\times\RR^{n_N}\times\RR^{m_1}\times
\cdots\times\RR^{m_M}\\
z_{b,k}=(z_{1,k},\ldots,z_{N,k},z_{N+1,k},\ldots,z_{N+M,k})\\
\hfill\in
\RR^{n_1}\times\cdots\times\RR^{n_N}\times\RR^{m_1}\times
\cdots\times\RR^{m_M}\\
d_{s,k}=(\delta_{1,k},\ldots,\delta_{N,k})\in\RR^N\\
d_{b,k}=(d_{1,k},\ldots,d_{N,k},d_{N+1,k},\ldots,d_{N+M,k})\\
\hskip 33mm\in
\RR^{n_1}\times\cdots\times\RR^{n_N}\times\RR^{m_1}\times
\cdots\times\RR^{m_M}.
\end{cases}
\end{equation}

\begin{algorithm}
\label{algo:1}
Let $\gamma\in\RPP$, $\varepsilon\in\zeroun$, 
$x_{s,0}\in\RR^{N}$, $x_{b,0}\in\RR^{p}$, 
$h_{s,0}\in\RR^{N}$, and 
$h_{b,0}\in\RR^{n+m}$. 
Iterate
\begin{equation}
\label{e:9hyR09}
\begin{array}{l}
\text{for}\;k=0,1,\ldots\\
\left\lfloor
\begin{array}{l}
\mu_k\in[\varepsilon,2-\varepsilon]\\
q_{s,k}=x_{s,k}-h_{s,k}\\
q_{b,k}=Ax_{b,k}-{h}_{b,k}\\
s_k=x_{s,k}-q_{s,k}/2\\
b_k=x_{b,k}-Qq_{b,k}\\
c_{s,k}=\proj_D(2s_k-x_{s,k})\\
c_{b,k}=W(2b_k-x_{b,k})\\
x_{s,k+1}=x_{s,k}+\mu_k(c_{s,k}-s_k)\\
x_{b,k+1}=x_{b,k}+\mu_k(c_{b,k}-b_k)\\
\text{for}\;i=1,\ldots,N\\
\left\lfloor
\begin{array}{l}
z_{i,k}=X_ib_k\\
(\delta_{i,k},d_{i,k})
=(0,y_i)+\prox_{\gamma\widetilde{\varphi}_i}
(2\sigma_{i,k}-\eta_{i,k},2z_{i,k}-h_{i,k}-y_i)\\
\end{array}
\right.\\[2mm]
\text{for}\;i=1,\ldots,M\\
\left\lfloor
\begin{array}{l}
z_{N+i,k}=L_ib_k\\
d_{N+i,k}=\prox_{\gamma{\psi}_i}(2z_{N+i,k}-h_{N+i,k})\\
\end{array}
\right.\\[2mm]
h_{s,k+1}=h_{s,k}+\mu_k(d_{s,k}-s_k)\\
h_{b,k+1}=h_{b,k}+\mu_k(d_{b,k}-z_{b,k}).
\end{array}
\right.\\[2mm]
\end{array}
\end{equation}
\end{algorithm}

\begin{proposition}
\label{p:1}
Consider the setting of Problem~\ref{prob:1}.
Suppose that 
\begin{equation}
\label{e:coer}
\lim_{\substack{s\in D,\; b\in E\\
\|s\|_2+\|b\|_2\to\pinf}}\mathsf{g}(s,Ab)=\pinf
\end{equation}
and that
\begin{equation}
\label{e:cq}
\big(D\times A(E)\big)\cap\reli\dom\mathsf{g}\neq\emp.
\end{equation}
Then Problem~\ref{prob:1} has at least one solution. Now let 
$(s_k)_{k\in\NN}$ and $(b_k)_{k\in\NN}$ be sequences generated by
Algorithm~\ref{algo:1}. Then
$(s_k)_{k\in\NN}$ converges to some $s\in\RR^N$ and 
$(b_k)_{k\in\NN}$ converges to some $b\in\RR^p$ such that 
$(s,b)$ solves Problem~\ref{prob:1}.
\end{proposition}
\begin{proof}
See Appendix~A.
\end{proof}

In most practical situations, \eqref{e:coer} and \eqref{e:cq} are 
typically satisfied. For example the following describes a scenario
that will be encountered in Section~\ref{sec:4}.

\begin{proposition}
\label{p:2}
Consider the setting of Problem~\ref{prob:1} and 
suppose that the following additional properties hold:
\begin{enumerate}
\item
\label{p:2i}
For every $i\in\{1,\ldots,N\}$, 
$\varphi_i=\theta_i+\alpha_i$, where 
$\theta_i\colon\RR^{n_i}\to\RP$ is convex and $\alpha_i\in\RPP$. 
\item
\label{p:2ii}
For every $i\in\{1,\ldots,M\}$, 
$\psi_i\colon\RR^{m_i}\to\RP$.
\item
\label{p:2iii}
For some $j\in\{1,\ldots,M\}$, $\psi_j(L_jb)\to\pinf$ as 
$\|b\|_2\to\pinf$ while $C^\top b=0$.
\item
\label{p:2iv}
$D\cap\RPP^N\neq\emp$.
\end{enumerate}
\end{proposition}
Then \eqref{e:coer} and \eqref{e:cq} are satisfied.
\begin{proof}
See Appendix~B.
\end{proof}

\subsection{Model selection}
\label{subsec:modsel}
In the context of log-contrast regression, a number of different
model selection strategies have been proposed, including stability
selection \cite{Lin2014,Meinshausen2010} and Generalized
Information Criteria \cite{Sun2018}. In \cite{Shi2016}, a
scale-dependent tuning parameter has been derived where the optimal
scale has been found via line search. Our joint scale and
regression modeling approach makes this line search obsolete, thus
yielding a parameter-free model selection scheme. More
specifically, we consider two model selection schemes. Firstly,
following \cite{Shi2016}, we consider
\begin{equation}
\label{e:lam0}
\lambda_0=\sqrt{2}q_n(r/p), 
\end{equation}
where $q_n(t)=n^{-1/2}\Phi^{-1}(1-t)$,
$\Phi^{-1}$ is the quantile function for the standard
normal distribution, and $r$ is the solution to the equation 
$r=q^4_1(r/p)+2q^2_1(r/p)$. In practice, this data-independent
model selection scheme may lead to inclusion of spurious 
coefficients. 
To assess the robustness of the inferred solutions 
we combine this theoretically derived regularization with 
stability selection \cite{Meinshausen2010}. The
original stability selection approach selects, for every subsample,
a small set of predictors from the regularization path,
e.g., the first $q$ predictors that appear along the path or the $q$
coefficients that are largest in absolute value across the entire
path.  We here propose to select, for every subsample, the nonzero
coefficients present at regularization parameter $\lambda_0$. Note
that $\lambda_0$ is sample-size dependent and hence needs to be
adapted to the specific subsample size used in stability selection.
As default values, we consider a subsample size of $\lceil n/2
\rceil$ and generate 100 subsamples.  The key diagnostic in
stability selection is the selection frequency profile for each
coefficient. To select a stable set of coefficients, a threshold
parameter $t_s \in [0.6, 0.9]$ is recommended
\cite{Meinshausen2010}, where all the coefficients with selection
frequency above $t_s$ are included in the final model. 


\section{Applications to compositional microbiome data}
\label{sec:4}
We apply several instances of the general log-contrast model
outlined in Problem~\ref{prob:1} in the context of microbiome data
analysis tasks. We set
$M=1$, $m_1=m$, $L_1=\Id$, and employ as a
regularization function the $\ell^1$ norm $\psi_1=\lambda
\|\cdot\|_1$. We use the functions in
Examples~\ref{ex:1} and \ref{ex:2} as instances of
the perspective loss functions $\widetilde{\varphi}_1$.
We refer to these instances as Least Squares (LS) and Huber 
log-contrast model, respectively. Thus, in case of the LS model, 
\eqref{e:prob0} becomes
\begin{equation}
\label{e:prob1}
\minimize{\sigma\in\RR,\:b\in E}
{\widetilde{\|\cdot\|_2}
\big(\sigma,Xb-y\big)+\lambda\|b\|_1},
\end{equation}
while in the case of the Huber model it becomes
\begin{multline}
\label{e:prob2}
\minimize{s\in D,\:b\in E}{\Sum_{i=1}^{n}\widetilde{h_\rho}
\big(\sigma_i,x_ib-\eta_i\big)+\lambda\|b\|_1},\\
\text{where}\quad
D=\menge{(\sigma,\ldots,\sigma)\in\RR^n}{\sigma\in\RR}.
\end{multline}
Note that the projection of a vector $s\in\RR^n$ onto $D$, as
required in Algorithm~\ref{algo:1}, is given by
\begin{equation}
\proj_Ds=\Bigg(\frac{1}{n}\sum_{i=1}^n\sigma_i,\ldots,
\frac{1}{n}\sum_{i=1}^n\sigma_i\Bigg).
\end{equation}
Dependent on the application, we use different
zero-sum constraints on $b$ as specified by the matrix $C$.  To
solve the various instances of Problem~\ref{prob:1}, we use 
Algorithm~\ref{algo:1} and set the parameter $\mu_k=1.9$ and
$\gamma=1$. We consider that the algorithm has converged when 
$\|b_k-b_{k+1}\|_2<10^{-6}$.

\subsection{Body mass index prediction from gut microbiome data}
We first consider a cross-sectional study that examines the
relationship between diet and gut microbiome composition, where
additional demographic covariates, including body mass index (BMI)
are available, referred to as COMBO data set \cite{Wu2011c}. After
pre-processing and filtering, the data set comprises the
log-transformed relative abundances of $p=87$ taxa at the genus
level across $n=96$ healthy subjects. Following previous analyses
\cite{Lin2014,Shi2016}, we
investigate the relationship between BMI and the microbial
compositions in a log-contrast regression framework. 
We use $C=\mathbf{1}_p$ to model the standard zero-sum constraint.
In addition to the compositional covariates, two covariate
measurements, fat and calorie intake, are also taken into account
via joint unpenalized least squares.  We investigate the influence
of different loss functions, LS and Huber, as well as the
sub-compositional constraints on the quality of the estimation, the
size of the support set, and the predictive power. 

\begin{figure}[h]
\centering
\includegraphics[width=12cm]{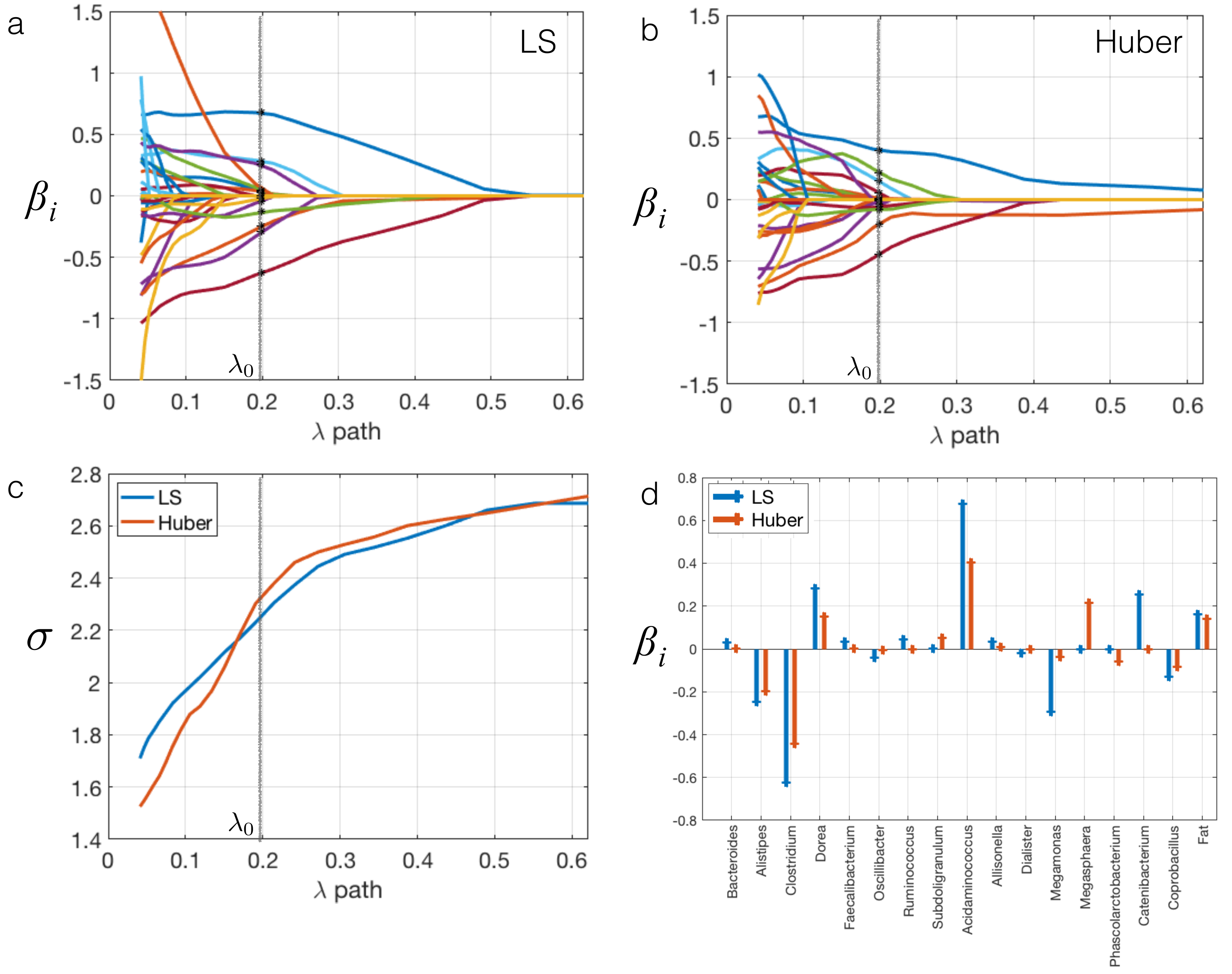}
\caption{a. Solution path of regression vector $b$ in the sparse 
Least Squares (LS) log-contrast model on the full COMBO data. 
The grey line marks the theoretical $\lambda_0$ from
\eqref{e:lam0}. b. Solution path of regression vector $b$ in
sparse Huber log-contrast model on the full COMBO data. c. Solution
path of the scale estimates $\sigma$ for both log-contrast models
on the full COMBO data. d. Comparison of the regression estimates
of both models at regularization parameter $\lambda_0$ on the full
data set. Both models agree on the two strongest predictors, 
the genera Clostridium and Acidaminococcus. 
}
\label{fig:Combo1}
\end{figure}

To highlight the ability of the algorithm to jointly estimate
regression and scale we solve the two problems over the
regularization path with $40$ $\lambda$ values on a log-linear grid
in $[0.0069,\ldots,0.6989]$. We also consider the theoretically
derived regularization parameter $\lambda_0=0.1997$ from
\eqref{e:lam0}. Figure~\ref{fig:Combo1}a and b show the solution
path of the regression vector $b$ for the sparse LS log-contrast
model and the Huber model, respectively. Figure~\ref{fig:Combo1}c
displays the corresponding joint scale estimates $\sigma$ for the LS
and the Huber model. The estimated regression coefficients at
$\lambda_0$ are highlighted in Figure~\ref{fig:Combo1}d. Both
models agree on a set of six genera, including Clostridium as
strongest negative and Acidaminococcus as the strongest positive
predictors. This implies that the log-ratio of Acidaminococcus to
Clostridium is positively associated with BMI. Other genera include
Alistipes, Megamonas, and Coprobacillus with negative coefficients,
and Dorea with positive coefficient. In \cite{Lin2014,Shi2016}, the
genera Alistipes, Clostridium, Acidaminococcus, and Allisonella
have been identified as key predictors. The solutions of the
perspective log-contrast models corroborates these finding for
Clostridium and Acidaminococcus, and to a less extent to Alistipes, 
whereas the genus Allisonella has only a small strictly
positive coefficient in both log-contrast models 
(Figure~\ref{fig:Combo1}d). 

Next, we consider the stability selection scheme introduced in
Section~\ref{subsec:modsel} with default parameters and threshold
$t_s=0.7$. 
Figure~\ref{fig:Combo2}a shows the stability-based frequency profile 
for the sparse LS and Huber log-contrast model. For both models,
only Clostridium and Acidaminococcus are selected. Stability
selection thus leads to a simple explanatory log-ratio model formed
by the ratio of the relative abundances of Acidaminococcus to
Clostridium. However, when considering the final model prediction
results, as shown in Figure~\ref{fig:Combo2}b for the Huber model,
this model can only explain normal to overweight participants (BMI
20-30) as $34$ out of 88 participants are considered outliers in
the Huber model. The overall refitted $R^2$ is $0.19$ under the
Huber model but increases to $0.43$ for the 60 inlier participants.

\begin{figure}[h]
\centering
\includegraphics[width=12cm]{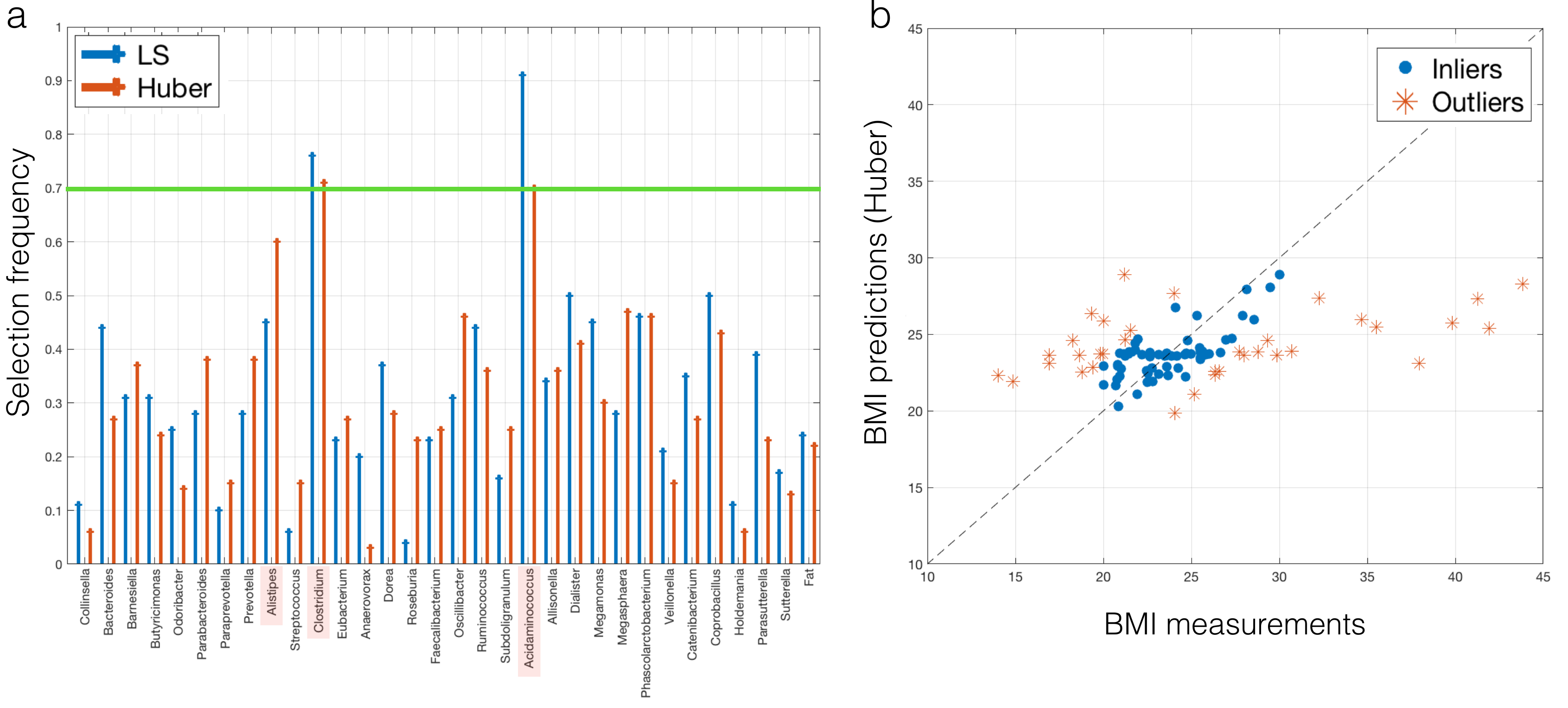}
\caption{a. Stability selection profile for all taxa selected with a 
frequency $>0.1$ in either the sparse Least Squares (LS, blue)
or the Huber (red) log-contrast model. The green dashed line
marks the stability threshold $t_\text{s}=0.7$, selecting the
genera Clostridium 
and Acidaminococcus. b. Prediction of BMI from the 
log-contrast of the two genera in the Huber log-contrast 
model for inliers (blue) and outliers
(red) (overall $R^2=0.19$). 
}
\label{fig:Combo2}
\end{figure}

Next, we investigate the influence of sub-compositional constraints
on the stability selection frequency for the two estimation
procedures.  We follow the analysis of \cite{Shi2016} and consider
a subset of 45 genera that have the highest relative abundances in
the data set. These 45 genera belong to $K=4$ distinct phyla:
Actinobacteria (two genera), Bacteroides (eight genera), Firmicutes
(32 genera), and Proteobacteria (three genera). The constraint
matrix $C$ is hence an orthogonal $(45 \times 4)$-dimensional
binary matrix. Figure~\ref{fig:Combo3}a and b show stability
selection profile for both the LS and  the Huber model with and
without compositional constraints, respectively.
Figure~\ref{fig:Combo3}c shows the difference in the selection
frequency profiles. Although several genera, including
Collinsella, Paraprevotella, Parabacteroides, Faecalibacterium,
Oscillibacter, and Parasutterela display significant frequency
differences, the two genera Clostridium and Acidaminococcus, both
belonging to the Firmicutes phylum, demonstrate again the highest
stability both with and without sub-compositional constraints.

\begin{figure}[h]
\centering
\includegraphics[width=12cm]{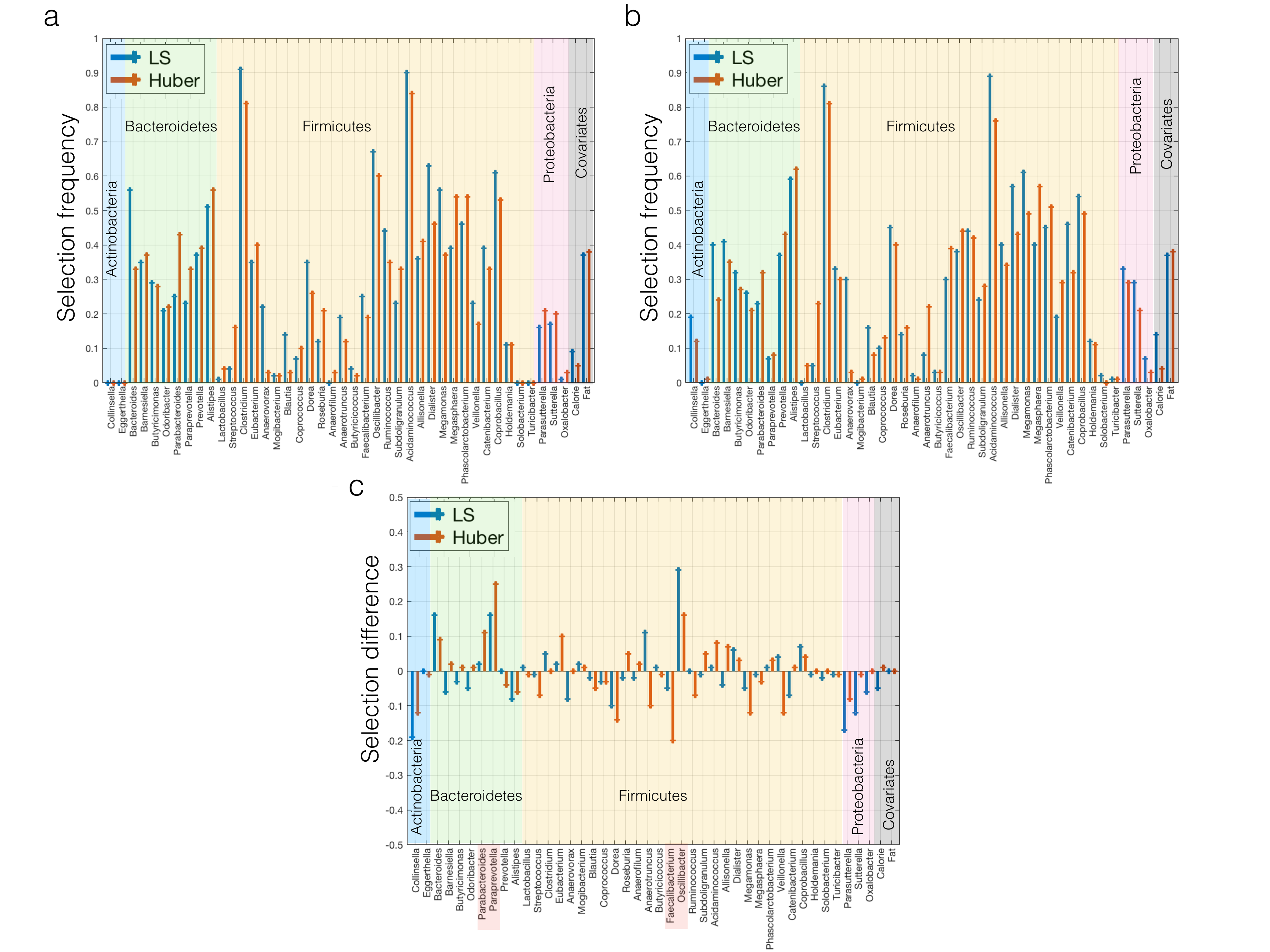}
\caption{a. Stability selection profiles for the subset of 45 taxa 
selected with a frequency $>0.1$ in either the sparse Least Squares
(LS, blue) or the Huber (red) log-contrast model with
sub-compositional constraints. b.  Same as a. but without
sub-compositional constraints. c. Stability selection frequency
differences between the two approaches. Several genera show
significant differences. The colors signify the different phyla
that the genera belong to and the noncompositional covariates fat
and diet intake. 
}
\label{fig:Combo3}
\end{figure}

\subsection{Relationship between soil microbiome and pH
concentration}
We next consider a dataset put forward in \cite{Lauber2009}
comprising $n=88$ soil samples from North and South America. Both
amplicon sequencing data and environmental covariates, including pH
concentrations, are available and have been re-analyzed via a
log-ratio approach in \cite{Morton2017a}. The amplicon data
contains $p=116$ OTUs, and we consider $C=\mathbf{1}_p$. 
We perform stability selection with 
default parameters as outlined in Section~\ref{subsec:modsel}. The
selection frequency of the different regression coefficients is
shown Figure~\ref{fig:pH}a. At stability threshold $t_\text{s}=0.7$,
seven taxa are selected in both models, four Acidobacteria, 
two Proteobacteria, and one Actinobacteria. After re-estimation of 
the two perspective log-contrast models on the selected subset,
the two taxa of order Ellin6513 and one taxon of family 
Koribacteraceae have negative coefficients whereas 
two taxa belonging to the
genus Balneimonas as well as one Rubrobacter taxon and one 
taxon of order RB41 have positive associations 
(Figure~\ref{fig:pH}b). This allows the model to be compactly
represented with only four log-ratios of compositions. The
coefficients show small variations with respect to the choice
of the data fitting term. The Huber model ($R^2=0.88$) deems 29
data points to be outliers in the final estimate
(Figure~\ref{fig:pH}c). For completeness, we include the mean
absolute deviation (MAD) between the model estimates and the data in
Figure~\ref{fig:pH}d. The selected taxa cover a wide range 
of average pH levels (as provided in \cite{Morton2017a}), ranging 
from $4.9$ to $6.75$, implying that the learned model may 
indeed generalize to other soil types not present 
in the current data set.

\begin{figure}[h]
\centering
\includegraphics[width=12cm]{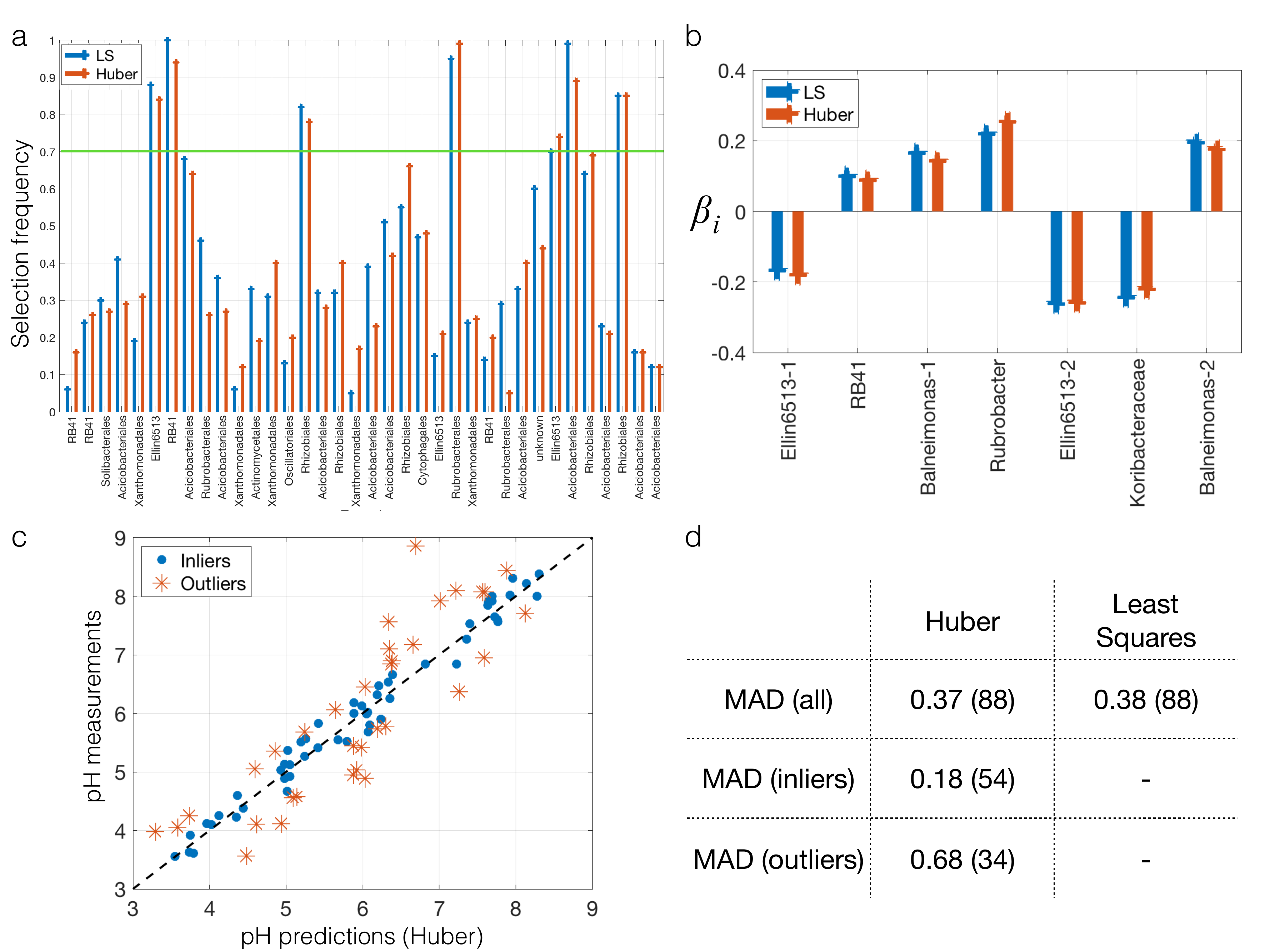}
\caption{a. Stability selection profile for all taxa selected with a 
frequency $>0.1$ in either the sparse Least Squares (LS, blue)
log-contrast model or the Huber model (red). The green dashed line
marks the stability selection threshold $t_\text{s}=0.7$. b.
Values of the seven selected log-contrast regression coefficients for
LS (blue) and the Huber (red) model. c. Prediction of pH
measurements from the Huber model for inliers (blue) and outliers
(red) ($R^2=0.88$). d. Table summarizing the mean absolute
deviation (MAD) of the two model estimates on the data. Numbers in
parentheses represent the number of inlier and outlier data points
under the Huber outlier model. 
}
\label{fig:pH}
\end{figure}

\section{Discussion and conclusion}
Finding linear relationships between continuous variables and
compositional predictors is a common task in many areas of science.
We have proposed a general estimation model
for high-dimensional log-contrast regression which includes many
previous proposals as special cases 
\cite{Lin2014,Mishra2019,Shi2016,Wang2017}.
Our model belongs to the class of perspective M-estimation models
which allows for scale estimation in the data fitting term while
preserving the overall convexity of the underlying model. This is
made possible due to recent advances in the theory of perspective 
functions \cite{Svva18,Combettes2018a,Combettes2018b}. We 
have introduced a
Douglas-Rachford algorithm that can solve the corresponding
constrained nonsmooth convex optimization problem with rigorous
guarantees
on the convergence of the iterates. We have illustrated the
viability of our approach on two microbiome data analysis tasks: 
body mass index (BMI) prediction from gut
microbiome data in the COMBO study and pH prediction from soil
microbial abundance data. Our joint regression and scale estimation
enabled the use of a universal single tuning parameter $\lambda_0$ 
\cite{Shi2016} to control the sparsity of the estimates. We have 
combined this approach with stability-based model selection 
\cite{Meinshausen2010} to determine sparse stable sets of 
log-contrast predictors. For the gut microbiome BMI analysis, the 
robust Huber log-contrast model identified two genera that can 
predict BMI well for normal to overweight participants 
while simultaneously identifying outliers with respect
to the log-contrast model. In the soil microbiome data set, we 
derived a parsimonious pH prediction model which requires 
only four log-ratios of microbial compositions with an overall
$R^2=0.88$. 

Going forward, we believe that the general log-contrast model 
and the associated optimization and model selection techniques 
presented here will provide a valuable off-the-shelf tool for 
log-contrast regression analysis when compositional data such 
as microbial relative abundance data are used as predictors 
in exploratory data analysis.


\section*{Appendix~A -- Proof of Proposition~\ref{p:1}}
Define $\mathsf{g}$ as in \eqref{e:g} and set
\begin{equation}
\label{e:20}
\begin{cases}
\mathsf{f}=\iota_{E\times D}\\
\mathsf{L}\colon\RR^N\times\RR^p\to\RR^N\times\RR^{n_1}
\times\cdots\times\RR^{n_N}\times\RR^{m_1}\times
\cdots\times\RR^{m_M}\\
\hskip 8mm (s,b)\;\;\;\mapsto (s,Ab)=
\big(s,X_1b,\ldots,X_Nb,L_1b,\ldots,L_Mb\big).
\end{cases}
\end{equation}
Then $\mathsf{f}\in\Gamma_0(\RR^{N+p})$ as the indicator of the
vector subspace $D\times E$, and
\begin{equation}
\label{e:prox1}
(\forall (s,b)\in\RR^{N+p})\quad\prox_{\gamma\mathsf{f}}(s,b)
=\big(\proj_Ds,\proj_Eb\big)
=\big(\proj_Ds,Wb\big),
\end{equation}
where the last identity follows from
\cite[Proposition~29.17(iii))]{Livre1}.
On the other hand, it follows from 
\cite[Proposition~2.3(ii)]{Svva18} and 
\cite[Proposition~8.6]{Livre1} that 
$\mathsf{g}\in\Gamma_0(\RR^{N+n+m})$. 
Furthermore, we derive 
from \cite[Propositions~24.11 and 24.8(ii)]{Livre1} that 
\begin{multline}
\label{e:prox2}
(\forall\big(s,u_1,\ldots,u_N,v_1,\ldots,v_M\big)\in
\RR^N\times\RR^{n_1}\times\cdots\times\RR^{n_N})\\
\prox_{\gamma\mathsf{g}}
(s,u_1,\ldots,u_N,v_1,\ldots,v_M)=\big((0,y_1)+
\prox_{\gamma\widetilde{\varphi}_1}(\sigma_1,u_1-y_1),\ldots\\
\ldots,(0,y_N)+
\prox_{\gamma\widetilde{\varphi}_N}(\sigma_N,u_N-y_N),
\prox_{\gamma\psi_1}v_1,\ldots,\prox_{\gamma\psi_M}v_M\big).
\end{multline}
In addition, \eqref{e:cq} implies that 
\begin{align}
\label{e:cq5}
\mathsf{L}(\dom\mathsf{f})\cap\dom\mathsf{g}
&=\big(\mathsf{L}(D\times E)\big)\cap
\dom\mathsf{g}\nonumber\\
&=\big(D\times A(E)\big)\cap\dom\mathsf{g}\nonumber\\
&\neq\emp.
\end{align}
Consequently, 
$\dom(\mathsf{f}+\mathsf{g}\circ\mathsf{L})\neq\emp$.
Thus, 
\begin{equation}
\mathsf{f}+\mathsf{g}\circ\mathsf{L}\in\Gamma_0(\RR^{N+p})
\end{equation}
while, using the variable $\mathsf{w}=(s,b)\in\RR^{N+p}$, 
\eqref{e:coer} and \eqref{e:20} imply that 
\begin{equation}
\label{e:coer3}
\lim_{\mathsf{w}\in\RR^{N+p}, \|\mathsf{w}\|_2\to\pinf}
\mathsf{f}(\mathsf{w})+\mathsf{g}(\mathsf{Lw})=\pinf.
\end{equation}
It therefore follows from 
\cite[Proposition~11.15(i)]{Livre1} that 
\begin{equation}
\label{e:a}
\text{Argmin}(\mathsf{f}+\mathsf{g}\circ\mathsf{L})\neq\emp.
\end{equation}
Since \eqref{e:prob0} is equivalent to 
\begin{equation}
\label{e:prob1}
\minimize{\mathsf{\mathsf{w}}\in\RR^{N+p}}
{\mathsf{f}(\mathsf{w})+\mathsf{g}(\mathsf{L}\mathsf{w})},
\end{equation}
we infer from \eqref{e:a} that Problem~\ref{prob:1} admits at 
least one solution. Note that \eqref{e:prob1} can be rewritten as
\begin{equation}
\label{e:prob2}
\minimize{\substack{
\mathsf{\mathsf{w}}\in\RR^{N+p}\\
\mathsf{\mathsf{z}}\in\RR^{N+n+m}\\
\mathsf{L}\mathsf{w}=\mathsf{z}}}
{\mathsf{f}(\mathsf{w})+\mathsf{g}(\mathsf{z})}.
\end{equation}
Now set
$\boldsymbol{u}=(\mathsf{w},\mathsf{z})\in\HHH=\RR^{2N+m+n+p}$ 
and 
\begin{equation}
\label{e:p3}
\begin{cases}
\boldsymbol{F}\colon\HHH\to\RX\colon(\mathsf{w},\mathsf{z})\mapsto
\mathsf{f}(\mathsf{w})+\mathsf{g}(\mathsf{z})\\
\boldsymbol{G}=\iota_{\boldsymbol{V}},\quad\text{where}\quad 
\boldsymbol{V}=\menge{(\mathsf{x},\mathsf{h})\in\HHH}
{\mathsf{L}\mathsf{x}=\mathsf{h}}.
\end{cases}
\end{equation}
Then $\boldsymbol{F}\in\Gamma_0(\HHH)$, 
$\boldsymbol{G}\in\Gamma_0(\HHH)$, and
\eqref{e:prob2} is equivalent to 
\begin{equation}
\label{e:prob3}
\minimize{\boldsymbol{u}\in\HHH}{\boldsymbol{F}(\boldsymbol{u})
+\boldsymbol{G}(\boldsymbol{u})}.
\end{equation}
Moreover, we deduce from \eqref{e:cq} that
\begin{equation}
\label{e:cq2}
\mathsf{L}(\dom\mathsf{f})\cap\reli\dom\mathsf{g}
=\big(D\times A(E)\big)\cap\reli\dom\mathsf{g}
\neq\emp.
\end{equation}
Consequently, using standard relative interior calculus
\cite[Section~6]{Rock70}, \eqref{e:p3} yields
\begin{align}
\label{e:cq3}
\reli(\dom\boldsymbol{G})\cap\reli(\dom\boldsymbol{F})
&=\boldsymbol{V}\cap\reli\big(\dom \mathsf{f}\times
\dom\mathsf{g}\big)\nonumber\\
&=\boldsymbol{V}\cap\big(\reli\dom\mathsf{f}\times
\reli\dom\mathsf{g})\nonumber\\
&=\boldsymbol{V}\cap(\dom\mathsf{f}\times
\reli\dom\mathsf{g})\nonumber\\
&=\menge{(\mathsf{x},\mathsf{Lx})}
{\mathsf{x}\in\RR^{N+p}}\cap(\dom\mathsf{f}\times
\reli\dom\mathsf{g})\nonumber\\
&\neq\emp.
\end{align}
Therefore, given $\gamma\in\RPP$, $\varepsilon\in\zeroun$, 
$\boldsymbol{v}_0\in\HHH$, and a sequence $(\mu_k)_{k\in\NN}$
in $[\varepsilon,2-\varepsilon]$, the nonsmooth convex 
minimization problem \eqref{e:prob3} can be solved using the 
Douglas-Rachford algorithm 
\begin{equation}
\label{e:kj18}
\begin{array}{l}
\text{for}\;k=0,1,\ldots\\
\left\lfloor
\begin{array}{l}
\boldsymbol{u}_k=\prox_{\gamma{\boldsymbol{G}}}\boldsymbol{v}_k\\
\boldsymbol{w}_{k}=\prox_{\gamma{\boldsymbol{F}}}
(2\boldsymbol{u}_{k}-\boldsymbol{v}_k)\\
\boldsymbol{v}_{k+1}=\boldsymbol{v}_k+
\mu_k(\boldsymbol{w}_{k}-\boldsymbol{u}_{k}),
\end{array}
\right.\\[2mm]
\end{array}
\end{equation}
which produces a sequence $(\boldsymbol{u}_k)_{k\in\NN}$ that
converges to a solution to \eqref{e:prob3} 
\cite[Section~28.3]{Livre1}. Next, it follows from 
\cite[Proposition~24.11 and Example~29.19(i)]{Livre1} that
\begin{equation}
\label{e:kj88}
\begin{cases}
\prox_{\gamma\boldsymbol{F}}\colon (\mathsf{w},\mathsf{z})
\mapsto\big(\prox_{\gamma\mathsf{f}}\mathsf{w},
\prox_{\gamma\mathsf{g}}\mathsf{z}\big)\\
\prox_{\gamma\boldsymbol{G}}\colon
(\mathsf{x},\mathsf{h})\mapsto 
(\mathsf{w},\mathsf{Lw}),
\;\text{where}\;\;
\mathsf{w}=\mathsf{x}-\mathsf{L}^\top
\big(\Id+\mathsf{LL}^\top\big)^{-1}
(\mathsf{Lx}-\mathsf{h}).
\end{cases}
\end{equation}
Now define
\begin{equation}
\label{e:nancago8}
\mathsf{R}=\mathsf{L}^\top(\Id+\mathsf{LL}^\top)^{-1}
\quad\text{and}\quad
(\forall k\in\NN)\quad 
\begin{cases}
\boldsymbol{u}_k=(\mathsf{w}_k,\mathsf{z}_k)\\
\boldsymbol{v}_k=(\mathsf{x}_k,\mathsf{h}_k)\\
\boldsymbol{w}_k=(\mathsf{c}_k,\mathsf{d}_k).
\end{cases}
\end{equation}
Then we derive from \eqref{e:kj88} that,
given $\mathsf{x}_0\in\RR^{N+p}$ and 
$\mathsf{h}_0\in\RR^{N+n+m}$, \eqref{e:kj18} becomes
\begin{equation}
\label{e:hall62}
\begin{array}{l}
\text{for}\;k=0,1,\ldots\\
\left\lfloor
\begin{array}{l}
\mathsf{q}_k=\mathsf{Lx}_k-\mathsf{h}_k\\
\mathsf{w}_{k}=\mathsf{x}_k-\mathsf{Rq}_k\\
\mathsf{z}_{k}=\mathsf{L}\mathsf{w}_{k}\\
\mathsf{c}_{k}=\prox_{\gamma\mathsf{f}}
(2\mathsf{w}_k-\mathsf{x}_k)\\
\mathsf{d}_{k}=\prox_{\gamma\mathsf{g}}
(2\mathsf{z}_k-\mathsf{h}_k)\\
\mathsf{x}_{k+1}=\mathsf{x}_k+\mu_k(\mathsf{c}_k-
\mathsf{w}_k)\\
\mathsf{h}_{k+1}=\mathsf{h}_k+\mu_k(\mathsf{d}_k-\mathsf{z}_k).
\end{array}
\right.\\[2mm]
\end{array}
\end{equation}
Let us partition the vectors appearing in \eqref{e:hall62}
according to their scale and regression components as
\begin{equation}
\label{e:nancago11}
(\forall k\in\NN)\quad
\begin{cases}
\mathsf{x}_k=(x_{s,k},x_{b,k})\in\RR^N\times\RR^p\\ 
\mathsf{h}_k=(h_{s,k},h_{b,k})\in\RR^N\times\RR^{n+m}\\
\mathsf{q}_k=(q_{s,k},q_{b,k})\in\RR^N\times\RR^{n+m}\\
\mathsf{w}_k=(s_k,b_k)\in\RR^N\times\RR^p\\
\mathsf{z}_k=(s_k,z_{b,k})\in\RR^N\times\RR^{n+m}\\ 
\mathsf{c}_k=(c_{s,k},c_{b,k})\in\RR^N\times\RR^p\\ 
\mathsf{d}_k=(d_{s,k},d_{b,k})\in\RR^N\times\RR^{n+m}.
\end{cases}
\end{equation}
In terms of these new variables, 
using the matrix $Q$ of \eqref{e:nancago12}, 
\eqref{e:20} and \eqref{e:nancago8} yield
\begin{equation}
\label{e:trenet2}
(\forall k\in\NN)\quad
\mathsf{R}\mathsf{q}_k=\big(q_{s,k}/2,Qq_{b,k}\big),
\end{equation}
and it follows from \eqref{e:nancago12}, \eqref{e:trenet}, 
\eqref{e:prox1}, \eqref{e:prox2}, and \eqref{e:trenet2}
that \eqref{e:hall62} is precisely \eqref{e:9hyR09}. Altogether,
since $(\boldsymbol{u}_k)_{k\in\NN}=
(\mathsf{w}_k,\mathsf{z}_k)_{k\in\NN}$ 
converges to a solution to \eqref{e:prob3}, 
$(\mathsf{w}_k)_{k\in\NN}=(s_k,b_k)_{k\in\NN}$ converges 
to a solution to Problem~\ref{prob:1}.

\section*{Appendix~B -- Proof of Proposition~\ref{p:2}}

\begin{itemize}
\item
If $s\notin\RP^N$, then 
\eqref{e:perspective} yields $(\forall b\in\RR^p)$ 
$\mathsf{g}(s,Ab)=\pinf$. On the other hand if, for some
$i\in\{1,\ldots,N\}$, $\sigma_i\in\RPP$ then
we deduce from \ref{p:2i} that $(\forall b\in\RR^p)$
$\widetilde{\varphi}_i(\sigma_i,X_ib-y_i)=
\sigma_i\theta_i((X_ib-y_i)/\sigma_i)+\alpha_i\sigma_i
\geq\alpha_i\sigma_i\to\pinf$ as $\sigma_i\to\pinf$.
Hence, \ref{p:2ii} entails that $(\forall b\in\RR^p)$ 
$\mathsf{g}(s,Ab)\to\pinf$ as $\|s\|_2\to\pinf$ while $s\in\RP^N$.
On the other hand, it follows from \ref{p:2iii} that
$(\forall s\in\RR^N)(\forall b\in E)$ 
$\mathsf{g}(s,Ab)\geq\psi_j(L_jb)\to\pinf$ as $\|b\|_2\to\pinf$.
Altogether, \eqref{e:coer} holds.
\item
It follows from \ref{p:2i} and \eqref{e:perspective} that
$(\forall i\in\{1,\ldots,N\})$ 
$\reli\dom\widetilde{\varphi}_i=\RPP\times\RR^{n_i}$.
Furthermore, \ref{p:2ii} yields
$(\forall i\in\{1,\ldots,M\})$ 
$\reli\dom{\psi}_i=\RR^{m_i}$. Therefore
$\reli\dom\mathsf{g}=\RPP^N\times\RR^n\times\RR^m$.
Since trivially $A(E)\subset\RR^{n+m}$, \eqref{e:cq} reduces
to \ref{p:2iv}.
\end{itemize}
\end{document}